\documentclass[12pt]{article}
\usepackage[a4paper]{geometry}
\usepackage[italian,english]{babel}
\usepackage[T1]{fontenc}
\usepackage[utf8]{inputenc} 

\usepackage{amsmath, accents}
\usepackage{amsfonts}
\usepackage{amstext}
\usepackage{amssymb}
\usepackage{amsthm}
\usepackage{amscd}
\usepackage{mathrsfs}
\usepackage{bbold}
\usepackage{dsfont}
\usepackage{bbm}

\usepackage[pagebackref,draft=false]{hyperref}
\hypersetup{colorlinks,
linkcolor=myrefcolor,
citecolor=mycitecolor,
urlcolor=myurlcolor}

\usepackage{cite}
\usepackage{caption}
\usepackage{etaremune}

\usepackage{xcolor}
\definecolor{myurlcolor}{rgb}{0,0,0.4}
\definecolor{mycitecolor}{rgb}{0,0.5,0}
\definecolor{myrefcolor}{rgb}{0.5,0,0}
\usepackage{graphicx}
\usepackage{tikz}
\usepackage{tikz-cd}
\usetikzlibrary{graphs,decorations.pathmorphing,decorations.markings}

\usepackage{verbatim}
\usepackage{lipsum}
\usepackage{etoolbox}
\usepackage{makeidx}
\usepackage{sectsty}
\usepackage{dsfont}
\usepackage{enumitem} 
\usepackage[]{latexsym}
\usepackage{braket}
\usepackage{caption}
\usepackage[utf8]{inputenx}
\usepackage{lmodern}
\usepackage{textcomp}
\usepackage{totcount}
\usepackage{blindtext}

\usepackage[capitalize, nameinlink]{cleveref} 

\newtheorem{theorem}{Theorem}[section]

\newtheorem{remark}[theorem]{Remark}

\newtheorem{definition}[theorem]{Definition}

\newtheorem*{proof*}{Proof}

\newcommand{\be}{\begin{equation}}
\newcommand{\ee}{\end{equation}}
\newcommand{\bea}{\begin{eqnarray}}
\newcommand{\eea}{\end{eqnarray}}

\newcommand{\dd}{{\rm d}}
\newcommand{\de}{\partial}

\usepackage{changebar}

\usepackage{mathtools}

\title{A coisotropic embedding theorem for pre-multisymplectic manifolds}

\author{L. Schiavone$^{1,2}$ \href{https://orcid.org/0000-0002-1817-5752}{\includegraphics[scale=0.7]{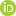}} \\ 
\footnotesize{$^{1}$\textit{Dipartimento di Matematica e Applicazioni Renato Caccioppoli, Università degli Studi di Napoli Federico II}}  \\
\footnotesize{$^{2}$\textit{ e-mail: \texttt{luca.schiavone@unina.it}}}
}

\date{}

\begin{document}

\maketitle

\begin{abstract}
\noindent Motivated by applications in Geometric Field Theory, we prove a coisotropic embedding theorem à là Gotay \cite{Gotay-CoisotropicEmbedding-1982} for pre-multisymplectic manifolds.
\end{abstract}

\tableofcontents

\section*{Introduction}

Multisymplectic geometry has emerged as a fundamental framework for the geometric description of Field Theories, extending the role that symplectic geometry plays in Classical Mechanics (see \cite{Tulczyjew-Kjowksi-SymplecticFramework-1979, Binz-Sniatycki-Fisher-ClassicalFields-1988, Carinena-Crampin-Ibort-Multisymplectic-1991, Roman-Roy-Multisymplectic-2009, Krupka-VariationalGeometry-2015} and references therein).
In this setting, the tangent and cotangent bundles of a configuration space are replaced by the first-order jet bundle of a fibration and its dual, providing a natural arena for analyzing the equations of motion and symmetries of Field Theories. 
These jet bundles are examples of pre-multisymplectic manifolds, which generalize the notion of pre-symplectic manifold involving higher-degree forms.
In this framework, to any Lagrangian (or Hamiltonian) function, a pre-multisymplectic structure is associated, which allows for an intrinsic formulation of a variational principle providing the equations of motion of the theory.

\noindent In previous contributions \cite{Ciaglia+-JacobiParticles-2020, Ciaglia+-JacobiFields-2020, Ciaglia-DiCosmo-Ibort-Marmo-Schiavone-Zampini-Peierls1-2024, Ciaglia-DiCosmo-Ibort-Marmo-Schiavone-Zampini-Peierls2-2022, Ciaglia-DiCosmo-Ibort-Marmo-Schiavone-Zampini-Peierls3-2022}, we showed that, at least locally near a Cauchy hypersurface, the equations of motion can be formulated as (infinite-dimensional) pre-symplectic Hamiltonian systems. 
In this setting, the space of solutions inherits a pre-symplectic structure whose kernel reflects the kernel of the underlying pre-multisymplectic form, encoding the gauge symmetries of the theory.

\noindent To define a Poisson structure on the space of solutions, we employed a regularization technique based on \textit{M. Gotay}'s \textit{coisotropic embedding theorem} \cite{Gotay-CoisotropicEmbedding-1982} (see also \cite{Guillemin-Sternberg-SymplecticTechniques-1990} for an equivariant version of the theorem and \cite{Oh-Park-DeformationsCoisotropic-2005} for a more modern approach).
Gotay's theorem provides a powerful tool in symplectic geometry, ensuring that any pre-symplectic manifold can be embedded as a coisotropic submanifold into a symplectic manifold, referred to as its \textit{symplectic thickening}. 
In \cite{Ciaglia-DiCosmo-Ibort-Marmo-Schiavone-Zampini-Peierls1-2024, Ciaglia-DiCosmo-Ibort-Marmo-Schiavone-Zampini-Peierls2-2022, Ciaglia-DiCosmo-Ibort-Marmo-Schiavone-Zampini-Peierls3-2022}, we used Gotay's theorem to coisotropically embed the space of solutions of the equations of motion of Classical Field Theories into a symplectic manifold where a Poisson structure is canonically defined. 
We then studied whether and how this Poisson structure projects back to a Poisson structure on the space of solutions. 

\noindent Building on the observation that Field Theories can locally be expressed as (infinite-dimensional) pre-symplectic Hamiltonian systems, the author has also used the coisotropic embedding theorem to establish a one-to-one correspondence between symmetries and conserved quantities \cite{Ciaglia-DiCosmo-Ibort-Marmo-Schiavone-Zampini-Symmetry-2022}, and to address the inverse problem of the calculus of variations for a class of implicit differential equations \cite{Schiavone-InverseProblemElectrodynamics-2024, Schiavone-InverseProblemImplicit-2024}.

\noindent While effective in all the cases mentioned above, this regularization approach requires handling infinite-dimensional spaces, as the space of solutions is typically infinite-dimensional. 
To circumvent the technical challenges associated with infinite-dimensional spaces, it may be useful to develop a multisymplectic analogue of Gotay's theorem, allowing coisotropic regularization to be carried out directly at the finite-dimensional level of the underlying pre-multisymplectic manifold.

\noindent The aim of this manuscript is to prove a coisotropic embedding theorem for pre-multisymplectic manifolds. 
Specifically, we will show that any pre-multisymplectic manifold can be embedded as a coisotropic submanifold into a larger multisymplectic manifold. 

\noindent It is worth pointing out that recent work has focused on the interplay between coisotropic submanifolds and multisymplectic geometry, particularly in the opposite direction taken in this manuscript, namely that of multisymplectic reduction. 
In particular, in \cite{deLeon-Izquierdo-CoisotropicReduction-2024} the authors analyze coisotropic submanifolds of pre-multisymplectic manifolds and study conditions under which they inherit multisymplectic structures.

\noindent The structure of this paper is as follows. 
In \cref{Sec: Preliminaries}, we recall preliminary notions about pre-multisymplectic manifolds that will be used throughout the manuscript. 
\cref{Sec: Coisotropic embeddings of pre-multisymplectic manifolds} is devoted to the statement and proof of the main theorem, and \cref{Sec: Application} provides an example to illustrate the main construction of the theorem.

\section{Preliminaries}
\label{Sec: Preliminaries}

\noindent Given a smooth $d$-dimensional differential manifold $\mathcal{M}$, we will usually denote by
\be
\left\{\, q^j \,\right\}_{j=1,...,d}
\ee
a system of local coordinates on it.

\noindent Recall the following definitions.

\begin{definition}[\textsc{Multisymplectic manifold}]
A $k$-\textbf{plectic manifold} (or \textbf{multisymplectic manifold}) is a smooth differential manifold $\mathcal{M}$ equipped with a closed and non-degenerate differential $k$-form $\omega$.
\end{definition}

\begin{definition}[\textsc{Pre-multisymplectic manifold}]
A \textbf{pre-}$k$-\textbf{plectic manifold} (or \textbf{ pre-multi-symplectic manifold}) is a smooth differential manifold $\mathcal{M}$ equipped with a closed differential $k$-form $\omega$.
\end{definition}

\begin{definition}[\textsc{$\ell$-multisymplectic orthogonal} \cite{Cantrijn-Ibort-DeLeon-Premultisymplectic-1999}]
Given a multisymplectic manifold $(\mathcal{M},\,\omega)$ and a submanifold $\mathcal{N} \subset \mathcal{M}$, the $\ell$-\textbf{multisymplectic orthogonal} of $\mathcal{N}$ in $\mathcal{M}$ at $n\in \mathcal{N}$ is
\be \label{Eq: ell-perp}
\mathbf{T}_n\mathcal{N}^{\perp,\,\ell} \,=\,\left\{\, V \in \mathbf{T}_n\mathcal{M} \;\;:\;\; i_{V \wedge W_1 \wedge ... \wedge W_\ell} \omega \,=\, 0 \,,\;\;\; \forall \,\, W_1,\,...,\,W_\ell \in \mathbf{T}_n\mathcal{N} \,\right\} \,,
\ee
where, with a slight abuse of notation, the tangent vectors $W_1,\,...,\,W_\ell$ in the contraction above denote the tangent vectors $\mathfrak{i}_\star W_1,\,...,\,\mathfrak{i}_\star W_\ell$, $\mathfrak{i}$ denoting the canonical embedding of $\mathcal{N}$ into $\mathcal{M}$.
\end{definition}

\noindent Evidently
\be
\mathbf{T}_n\mathcal{N}^{\perp,\,j} \subseteq \mathbf{T}_n\mathcal{N}^{\perp,\,i} \,, \;\;\; \text{if } \,\, i > j \,,
\ee
and
\be
\mathbf{T}_n\mathcal{N}^{\perp,\,\ell} \,=\, \mathbf{T}_n \mathcal{M} \,,\;\;\; \text{if } \,\, \ell > \operatorname{dim} \mathbf{T}_n \mathcal{N} \,.
\ee

\begin{definition}[\textsc{$\ell$-Coisotropic submanifold} \cite{Cantrijn-Ibort-DeLeon-Premultisymplectic-1999}] \label{Def: ell-coisotropic submanifold}
A submanifold $\mathcal{N}$ of a multisymplectic manifold $(\mathcal{M},\, \omega)$ is $\ell$-coisotropic if $\mathbf{T}_n\mathcal{N}^{\perp,\, \ell} \subseteq \mathbf{T}_n \mathcal{N}\,,\;\;\; \forall \,\, n \in \mathcal{N}$.
\end{definition}    

\begin{definition}[\textsc{Bundle of $k$-forms}]
Given a smooth differential manifold $\mathcal{M}$, we will denote by $\Lambda^{k}(\mathcal{M})$ the bundle of differential $k$-forms on $\mathcal{M}$, namely the vector bundle over $\mathcal{M}$ whose typical fibre at $m \in \mathcal{M}$ is $\bigwedge^{k} \mathbf{T}^\star_m \mathcal{M}$.
We will denote by $\pi$ the canonical projection $\pi \;:\;\; \Lambda^{k}(\mathcal{M}) \to \mathcal{M}$.
\end{definition}

\noindent We will usually denote by 
\be
\left\{\, q^j,\, p_{j_1...j_k} \,\right\}_{j,j_1,...,j_k=1,...,d}
\ee
a system of local coordinates on $\Lambda^{k}(\mathcal{M})$ adapted to the system of coordinates on $\mathcal{M}$ mentioned above, where the fibered coordinates $p_{j_1...j_k}$ have to be understood as the $k$-forms
\be
p_{j_1...j_k} \dd q^{j_1} \wedge ... \wedge \dd q^{j_k} \,.
\ee

\noindent Consider a pre-$k$-plectic manifold $(\mathcal{M},\, \omega)$.
The kernel of $\omega$ is, at each $m \in \mathcal{M}$, a subspace of $\mathbf{T}_m \mathcal{M}$, denoted by $K_m$.
We will always assume $\omega$ to have constant rank, so that $K_m$ and $K_n$ are isomorphic for any pair $m,\, n \in \mathcal{M}$.
Given that $\omega$ is closed and of constant rank, by means of \textit{Frobenius theorem} one shows that $K_m$ provides a completely integrable distribution on $\mathcal{M}$, so that there exists a foliation of $\mathcal{M}$ such that the tangent space to each leaf at each point $m \in \mathcal{M}$ coincides with $K_m$ and there exists a unique leaf passing through any point of $\mathcal{M}$.
We refer to \cite{Candel-Conlon-Foliations-2000} for the basic notions about foliations we will use along the manuscript.
Sometimes, we will consider on $\mathcal{M}$ a system of local coordinates adapted to such a foliation
\be
\left\{\, q^j \,\right\}_{j=1,...,d} \,=\, \left\{\, x^a,\, f^A \,\right\}_{a=1,...,l;A=1,...,r} \,,
\ee
where $l$ is the rank of $\omega$ and $r$ is the dimension of its kernel, so that $l + r = d$. 
Coordinates $x^a$ are a system of coordinates on the space of leaves of the foliation, namely, locally each leaf is a level set of the type
\be
\mathcal{F}_{c} \,=\, \left\{\, x^a \,=\, c^a \,, \;\;\; c^a \in \mathbb{R} \;\;\; \forall\,\, a=1,...,l \,\right\} \,,
\ee
whereas coordinates $f^A$ (for fixed values of the $x^a$'s) individuate a point on each leaf.
In this system of coordinates, $K_m$ reads
\be
K_m \,=\, \langle \left\{\, \frac{\de}{\de f^A}\,\right\}_{A=1,...,r}\rangle \,,
\ee
where the symbol $\langle \, \cdot \, \rangle$ denotes the linear span of the vectors appearing inside the angle brackets.

\noindent A complement to $K_m$ at each $m \in \mathcal{M}$ is not canonically defined.
Indeed, each choice of a complement $W_m$ such that
\be
\mathbf{T}_m \mathcal{M} \,=\, K_m \oplus W_m \,,
\ee
amounts to the choice of a \textit{connection} on $\mathcal{M}$.

\begin{definition}[\textsc{Connection on $\mathcal{M}$}]
A \textbf{connection} on $\mathcal{M}$ is an idempotent\footnote{By \textit{idempotent} we mean that $P^2 = P$.} smooth $(1,\,1)$-tensor field on $\mathcal{M}$ whose image, at each $m \in \mathcal{M}$, is $K_m$. 
\end{definition}

\noindent Locally, a connection can be written as
\be
P \,=\, (\dd f^A - P^A_a \dd x^a) \otimes \frac{\de}{\de f^A} \,,
\ee
where the functions $P^A_a$ are the so-called \textit{connection coefficients}.
By means of $P$, the tangent space to $\mathcal{M}$ at each point splits into the direct sum of tangent vectors in the image of $P$ and tangent vectors in the kernel of $P$.
Being a $(1,\,1)$-tensor field on $\mathcal{M}$, $P$ defines a map from vector fields on $\mathcal{M}$ to vector fields on $\mathcal{M}$ via contraction.
The image of such a map reads
\be
K_m \,=\, \langle\left\{\, V_A \,:=\, \frac{\de}{\de f^A}\biggr|_m\,\right\}_{A=1,...,r}\rangle \,,
\ee
and is usually referred to as the space of \textit{vertical vectors}, whereas the kernel reads
\be
H_m \,=\, \langle\left\{\, H_a \,:=\, \frac{\de}{\de x^a}\biggr|_m + P^A_a(m) \frac{\de}{\de f^A}\biggr|_m \,\right\}_{a=1,...,l}\rangle \,,
\ee
and is usually referred to as the space of \textit{horizontal vectors}.
The pointwise splitting $\mathbf{T}_m \mathcal{M} = K_m \oplus H_m$ extends to a global splitting of the tangent bundle $\mathbf{T}\mathcal{M} = \mathbf{V}^P \oplus \mathbf{H}^P$ because both $\mathbf{V}^P$ and $\mathbf{H}^P$ are smooth vector sub-bundles. 
The vertical bundle $\mathbf{V}^P$ corresponds to the kernel of the pre-multisymplectic form $\omega$. 
Since $\omega$ is a smooth form assumed to have constant rank, its kernel is a well-defined smooth sub-bundle. 
The horizontal bundle $\mathbf{H}^P$ is defined as the kernel of the connection $P$. 
By construction, $P$ is a smooth $(1,\,1)$-tensor field. 
As a projector onto a constant-rank bundle, $P$ also has constant rank, and thus its kernel $\mathbf{H}^P$ is also a smooth vector sub-bundle. 
The direct sum of two smooth sub-bundles that span the tangent bundle at every point constitutes a global vector bundle splitting.
Sections of $\mathbf{V}^P$ (resp. $\mathbf{H}^P$) are called \textit{vertical} (resp. \textit{horizontal}) vector fields.

\noindent Given any connection $P$, the related $(1,\,1)$-tensor field $R \,=\, \mathbf{1}-P$, where $\mathbf{1}$ denotes the identity operator, provides a complementary splitting to that induced by $P$, where the horizontal space coincides with the vertical space of $P$ and vice versa.
We will denote by 
\be
\mathbf{T}\mathcal{M} \,=\, \mathbf{V}^R \oplus \mathbf{H}^R \,, 
\ee
the splitting of $\mathbf{T}\mathcal{M}$ induced by $R$.

\noindent The splitting generated by $P$ on vector fields induces a splitting for differential $k$-forms on $\mathcal{M}$.
Indeed, given a $k$-form $\alpha$ at some point of $\mathcal{M}$, it can always be written as
\be
\alpha \,=\, \alpha_P^\parallel + \alpha_P^\perp \,,
\ee
where 
\be
\alpha^\parallel_P \,:=\, \alpha(\underbrace{P(\,\cdot\,),...,P(\,\cdot\,)}_{\text{k times}}) \,, 
\ee
and
\be
\alpha^\perp_P \,:=\, \alpha - \alpha^\parallel_P \,.
\ee
We will refer to $\alpha^\parallel_P$ and $\alpha^\perp_P$ as the \textit{parallel} and \textit{transversal} components of $\alpha$ associated with the connection $P$.

\noindent In the system of local coordinates chosen, if one writes the $k$-form in terms of the basis of differential $1$-forms
\be
\left\{\, \dd x^a,\, \dd f^A - P^A_a \dd x^a \,=:\, \theta^A \,\right\}_{a=1,...,l;A=1,...,r} \,,
\ee
dual to the basis of vertical and horizontal vector fields associated with $P$, namely such that
\be
\begin{split}
\dd x^a \left(\frac{\de}{\de x^b} + P^A_b \frac{\de}{\de f^A}\right) \,=\, \delta^a_b \,, \qquad \dd x^a \left(\, \frac{\de}{\de f^B}\,\right) \,=\, 0 \,, \\
\theta^A \left(\frac{\de}{\de x^b} + P^A_b \frac{\de}{\de f^A}\right) \,=\, 0 \,, \qquad \theta^A \left(\, \frac{\de}{\de f^B}\,\right) \,=\, \delta^A_B \,,
\end{split}
\ee
one has
\be
\begin{split}
\alpha \,=\, &\alpha_{A_1...A_k} \theta^{A_1} \wedge ... \wedge \theta^{A_k} + \\
&\alpha_{a_1 A_2 ... A_k} \dd x^{a_1} \wedge \theta^{A_2} \wedge ... \wedge \theta^{A_k} + \\ 
&\;\; \vdots \\
&\alpha_{a_1 ... a_k} \dd x^{a_1} \wedge ... \wedge \dd x^{a_k} \,. 
\end{split}
\ee
In this system of local coordinates, one gets
\be
\alpha^\parallel_P \,=\, \alpha_{A_1...A_k} \theta^{A_1} \wedge ... \wedge \theta^{A_k} \,,
\ee
and
\be
\alpha^\perp_P \,=\, \alpha_{a_1 A_2 ... A_k} \dd x^{a_1} \wedge \theta^{A_2} \wedge ... \wedge \theta^{A_k} + ... + \alpha_{a_1 ... a_k} \dd x^{a_1} \wedge ... \wedge \dd x^{a_k} \,.
\ee
Since this decomposition is defined pointwise using the smooth tensor field $P$, it is itself smooth. 
Building upon the established global splitting of the tangent bundle, this construction naturally defines a splitting of the bundle of $k$-forms into two smooth vector sub-bundles.
We will denote by
\be
\Lambda^k(\mathcal{M}) \,=\, {\Lambda^k}^\parallel_P(\mathcal{M}) \oplus {\Lambda^k}^\perp_P(\mathcal{M}) \,,
\ee
the splitting of the fibre bundle $\Lambda^k(\mathcal{M})$ into the direct sum of the fibre subbundles ${\Lambda^k}^\parallel_P(\mathcal{M})$ and ${\Lambda^k}^\perp_P(\mathcal{M})$ whose sections are parallel (resp. transversal) differential $k$-forms with respect to $P$.
We will denote by 
\be
\left\{\, q^j,\, p_{A_1...A_k} \,\right\}_{j=1,...,d;A_1,...,A_k=1,...,r} \,,
\ee
a system of local coordinates on ${\Lambda^k}^\parallel_{P}(\mathcal{M})$ and by
\be
\left\{\, q^j,\, p_{A_1 ... A_k},\, p_{a_1 A_2 ... A_k},\,...,\, p_{a_1 ... a_k} \,\right\}_{j=1,...,d;A_1,...,A_k=1,...,r} \,,
\ee
a system of local coordinates on ${\Lambda^k}^\perp_P(\mathcal{M})$.

\noindent Analogously, by means of the connection $R$, one can construct the splitting
\be
\alpha \,=\, \alpha^\parallel_R + \alpha^\perp_R \,,
\ee
where
\be
\alpha^\parallel_R \,:=\, \alpha(\underbrace{R(\,\cdot\,),...,R(\,\cdot\,)}_{\text{k times}}) \,, 
\ee
and
\be
\alpha^\perp_R \,:=\, \alpha - \alpha^\parallel_R \,.
\ee
In the system of local coordinates chosen
\be
\alpha^\parallel_R \,=\, \alpha_{a_1 ... a_k} \dd x^{a_1} \wedge ... \wedge \dd x^{a_k} \,,
\ee
while
\be
\begin{split}
\alpha^\perp_R \,=\, &\alpha_{A_1...A_k} \theta^{A_1} \wedge ... \wedge \theta^{A_k} + \\
&\alpha_{a_1 A_2 ... A_k} \dd x^{a_1} \wedge \theta^{A_2} \wedge ... \wedge \theta^{A_k} + \\
&\;\;\vdots \\
&\alpha_{a_1...a_{k-1}A_k} \dd x^{a_1} \wedge... \wedge \dd x^{a_{k-1}} \wedge \theta^{A_k} \,.
\end{split}
\ee

\noindent We will denote by
\be
\Lambda^k(\mathcal{M}) \,=\, {\Lambda^k}^\parallel_R(\mathcal{M}) \oplus {\Lambda^k}^\perp_R(\mathcal{M}) \,,
\ee
the splitting of the fibre bundle $\Lambda^k(\mathcal{M})$ into the direct sum of the fibre subbundles ${\Lambda^k}^\parallel_R(\mathcal{M})$ and ${\Lambda^k}^\perp_R(\mathcal{M})$ whose sections are parallel (resp. transversal) differential $k$-forms with respect to $R$.
Clearly, if $k > l$, ${\Lambda^k}^\parallel_R(\mathcal{M})$ is empty, and $\Lambda^k(\mathcal{M}) \,=\, {\Lambda^k}^\perp_R(\mathcal{M})$.

\noindent We will denote by
\be
\left\{\, q^j,\, p_{a_1...a_k} \,\right\}_{j=1,...,d;a_1,...,a_k=1,...,r} \,,
\ee
a system of local coordinates on ${\Lambda^k}^\parallel_{R}(\mathcal{M})$ and by
\be
\left\{\, q^j,\, p_{A_1 ... A_k},\, p_{a_1 A_2 ... A_k},\,...,\, p_{a_1 ... a_{k-1} A_k} \,\right\}_{j=1,...,d;a_k=1,...,r;A_k=1,...,l} \,,
\ee
a system of local coordinates on ${\Lambda^k}^\perp_R(\mathcal{M})$.

\section{Coisotropic embeddings of pre-multisymplectic manifolds}
\label{Sec: Coisotropic embeddings of pre-multisymplectic manifolds}

We are now ready to prove a coisotropic embedding theorem à là Gotay for pre-multisymplectic manifolds.

\begin{theorem} \label{Thm: coisotropic embedding}
Let $(\mathcal{M}, \omega)$ be a pre-$k$-plectic manifold, with $k > 2$.
Assume $K \,=\, \mathrm{ker}\, \omega$ to have constant rank.
Then, there exists a multisymplectic manifold $\tilde{\mathcal{M}}$, referred to as the \textbf{multisymplectic thickening} of $\mathcal{M}$, and an embedding $\mathfrak{i} \;:\;\; \mathcal{M} \hookrightarrow \tilde{\mathcal{M}}$, such that:
\begin{enumerate}
    \item $\tilde{\mathcal{M}}$ is equipped with a multisymplectic $k$-form $\tilde{\omega}$.
    \item $(\mathcal{M}, \omega)$ is a $(k-1)$-coisotropic submanifold of $(\tilde{\mathcal{M}}, \tilde{\omega})$ and $\mathfrak{i}^\star \tilde{\omega} \,=\, \omega$.
\end{enumerate}
\end{theorem}

\begin{proof}
\noindent Consider a $d$-dimensional pre-$k$-plectic manifold $(\mathcal{M},\,\omega)$.
In the system of coordinates adapted to the foliation induced by the kernel of $\omega$ described in \cref{Sec: Preliminaries}, $\omega$ reads:
\be
\omega \,=\, \omega_{a_1...a_k}(x) \dd x^{a_1} \wedge ... \wedge \dd x^{a_k} \,.
\ee

\noindent Consider the bundle ${\Lambda^{k-1}}^\perp_R(\mathcal{M})$ described in \cref{Sec: Preliminaries}, with the system of local coordinates 
\be
\left\{\, q^j,\, p_{A_1 ... A_k},\, p_{a_1 A_2 ... A_k},\,...,\, p_{a_1 ... a_{k-1} A_k} \,\right\}_{j=1,...,d;a_k=1,...,r;A_k=1,...,l} \,.
\ee
Let us also denote by $\tau: {\Lambda^{k-1}}^\perp_R(\mathcal{M}) \to \mathcal{M}$ the canonical bundle projection.
Consider the $(k-1)$-form on ${\Lambda^{k-1}}^\perp_R(\mathcal{M})$ defined by
\be
{\Theta_0^{(k-1)}}_p(V_1,...,V_{k-1}) \,=\, p(\tau_\star V_1,...,\tau_\star V_{k-1}) \,,
\ee
where $p$ has to be understood as a point in ${\Lambda^{k-1}}^\perp_R(\mathcal{M})$ on the left hand side and as a (transversal) $(k-1)$-covector on $\mathcal{M}$ on the right hand side, $\tau_\star$ is the push-forward of $\tau$, and $V_j \in \mathbf{T}_p {\Lambda^{k-1}}^\perp_R(\mathcal{M})\,,\;\;\; \forall \,\, j=1,...,k-1$.
In the system of local coordinates chosen
\be
\begin{split}
{\Theta_0^{(k-1)}} \,=\, &p_{A_1\,...\,A_{k-1}} \theta^{A_1} \wedge ... \wedge \theta^{A_{k-1}} + \\ 
&p_{a_1 \, A_2 \,...\, A_{k-1}} \dd x^{a_1} \wedge \theta^{A_2} \wedge ... \wedge \theta^{A_{k-1}} +\\
&\;\; \vdots \\
&p_{a_1 \,...\, a_{k-2} \, A_{k-1}} \dd x^{a_1} \wedge ... \wedge \dd x^{a_{k-2}} \wedge \theta^{A_{k-1}} \,.
\end{split}
\ee

\noindent Consider on ${\Lambda^{k-1}}^\perp_R(\mathcal{M})$ the following $k$-form
\be
\begin{split}
\tilde{\omega} \,:=\, \tau^\star \omega + \dd \Theta^{(k-1)}_0 \,.
\end{split}
\ee

\noindent The $k$-form $\tilde{\omega}$ is closed by construction.

\noindent To prove that it is also non-degenerate, we will consider its contraction along a vector field $X$ written in the basis
\[
\left\{\, H_a = \frac{\de}{\de x^a} + {P_x}^A_a \frac{\de}{\de f^A},\, V^A = \frac{\de}{\de f^A},\, \frac{\de}{\de p_{A_1 \,...\, A_{k-1}}},\, \frac{\de}{\de p_{a_1\,A_2\,...\,A_{k-1}}},\,...\,,\, \frac{\de}{\de p_{a_1\,...\,a_{k-2}\,A_{k-1}}} \,\right\} \,,
\]
namely
\be
\begin{split}
X \,=\, X^a H_a + X^A V_A + &X_{A_1 \,...\, A_{k-1}} \frac{\de}{\de p_{A_1 \,...\, A_{k-1}}} + \\
&X_{a_1 \, A_2 \,...\, A_{k-1}} \frac{\de}{\de p_{a_1\,A_2\,...\,A_{k-1}}} + \\
& \;\; \vdots \\
&X_{a_1\,...\,a_{k-2}\,A_{k-1}} \frac{\de}{\de p_{a_1\,...\,a_{k-2}\,A_{k-1}}} \,,
\end{split}
\ee
and prove that 
\[
i_X \tilde{\omega} \,=\, 0 \;\; \forall \,\, X \in \mathfrak{X}({\Lambda^{k-1}}^\perp_R(\mathcal{M}))\;\; \implies \;\; X = 0
\]
The contraction $i_X \widetilde{\omega}$ reads
\be \label{Eq: contraction multisymplectic tickening}
\begin{split}
i_X \widetilde{\omega} \,=\,  &{C^0}_{A_1 \,...\, A_{k-1}} \, \theta^{A_1} \wedge ... \wedge \theta^{A_{k-1}} + \\
&{C^1}_{a_1 \, A_2 \,...\, A_{k-1}} \, \dd x^{a_1} \wedge \theta^{A_2} \wedge \,...\, \wedge \theta^{A_{k-1}} + \\
&\;\; \vdots \\
&{C^{k-1}}_{a_1 \,...\, a_{k-1}} \, \dd x^{a_1} \wedge \,...\, \dd x^{a_{k-1}} \, + \\
& {D^0}^{A_1} \, \dd p_{A_1 \,...\, A_{k-1}} \wedge \theta^{A_2} \wedge \,...\, \wedge \theta^{A_{k-1}} + \\
&{D^1_0}^{a_1} \, \dd p_{a_1 \,A_2\,...\, A_{k-1}} \wedge \theta^{A_2} \wedge \,...\, \wedge \theta^{A_{k-1}} + \\ 
&{D^1_1}^{A_2} \dd p_{a_1 \, A_2 \,...\, A_{k-1}} \wedge \dd x^{a_1} \wedge \theta^{A_3} \wedge \,...\, \wedge \theta^{A_{k-1}} + \\
& \;\; \vdots \\
&{D^{k-1}_{k-2}}^{a_1} \dd p_{a_1 \,...\, a_{k-2} A_{k-1}} \wedge \dd x^{a_2} \wedge ... \wedge \dd x^{a_{k-2}} \wedge \theta^{A_{k-1}} + \\
&{D^{k-1}_{k-1}}^{A_{k-1}} \dd p_{a_1 \,...\, a_{k-2} A_{k-1}} \wedge \dd x^{a_1} \wedge ... \wedge \dd x^{a_{k-2}} \,. 
\end{split}
\ee
In particular, ${D^0}^{A_1}$ reads
\be
{D^0}^{A_1} \,=\, - X^{A_1} \,,
\ee
and, since
\be
\dd p_{A_1 \,...\, A_{k-1}} \wedge \theta^{A_2} \wedge \,...\, \wedge \theta^{A_{k-1}} 
\ee
is independent of all the other forms of the above decomposition, the condition $i_X \widetilde{\omega} \,=\, 0\,,\;\;\; \forall \,\, X$ gives
\be \label{Eq: XA=0}
X^A \,=\, 0 \,,\;\; \forall \,\, A=1,...,l \,.
\ee
The latter condition implies
\begin{align}
{D^1_1}^{A_2} \,=\, -X^{A_2} \,=\, 0 \,,\\
{D^1_0}^{a_1} \,=\, -X^{a_1} \,.
\end{align}
Again, given that
\be
\dd p_{a_1 \,A_2\,...\, A_{k-2}} \wedge \theta^{A_2} \wedge \,...\, \wedge \theta^{A_{k-1}}
\ee
is independent of all the other forms appearing in \eqref{Eq: contraction multisymplectic tickening}, the condition $i_X \widetilde{\omega} \,=\, 0\,,\;\;\; \forall \,\, X$ implies
\be \label{Eq: Xa=0}
X^a \,=\, 0 \,,\;\; \forall \,\, a=1,...,r \,.
\ee
Conditions \eqref{Eq: XA=0} and \eqref{Eq: Xa=0} imply in turn that all the ${D^j_k}$ vanish and that
\begin{align}
{C^0}_{A_1 \,...\, A_{k-1}} \,&=\, X_{A_1\,...\,A_{k-1}} \,, \\
{C^1}_{a_1\,A_2 \,...\,A_{k-1}} \,&=\, X_{a_1 \,A_2 \,...\,A_{k-1}} \,, \\
& \vdots \\
{C^{k-2}}_{a_1 \,...\, a_{k-2}\,A_{k-1}} \,&=\, X_{a_1 \,...\, a_{k-2}\,A_{k-1}} \,,\\
{C^{k-1}}_{a_1 \,...\, a_{k-1}} \,&=\, 0 \,.
\end{align}
Since the $C^j_k$ are coefficients of independent $(k-1)$-forms, the condition $i_X \widetilde{\omega} \,=\, 0 \,,\;\; \forall \,\, X$ implies
\begin{align}
X_{A_1\,...\,A_{k-1}} \,&=\, 0 \,, \label{Eq: Xmu0=0} \\
X_{a_1 \,A_2 \,...\,A_{k-1}}\,&=\,0 \,, \label{Eq: Xmu1=0} \\
& \vdots  \nonumber \\
X_{a_1 \,...\, a_{k-2}\,A_{k-1}} \,&=\, 0 \,. \label{Eq: Xmuk-2=0}
\end{align}
The chain of conditions \eqref{Eq: XA=0}, \eqref{Eq: Xa=0}, \eqref{Eq: Xmu0=0}, \eqref{Eq: Xmu1=0}, ..., \eqref{Eq: Xmuk-2=0} prove that 
\be
i_X \widetilde{\omega} \,=\, 0 \,,\;\; \forall\,\, X \in \mathfrak{X}({\Lambda^{k-1}}^\perp_R(M)) \;\; \implies \;\; X \,=\,0 \,,
\ee
namely, that $\widetilde{\omega}$ is non-degenerate, and, thus, multisymplectic.

\noindent Thus, the multisymplectic manifold we were searching for is $\tilde{\mathcal{M}} \,=\, {\Lambda^{k-1}}^\perp_R(\mathcal{M})$ and $\mathfrak{i}$ is the zero-section of $\tau$.

\noindent A straightforward computation also shows that
\be
\mathfrak{i}^\star \tilde{\omega} \,=\, \omega \,.
\ee

\noindent To prove that $\mathcal{M}$ is a $(k-1)$-coisotropic submanifold of ${\Lambda^{k-1}}^\perp_R(\mathcal{M})$, having in mind \cref{Def: ell-coisotropic submanifold}, consider the contraction $\tilde{\omega}_m(X,\, W_1,...,\,W_{k-1})$ for $m \in \mathcal{M}$, and $W_1,\, W_2,\,...,\, W_{k-1} \in \mathbf{T}_m \mathcal{M}$
\be
W_j \,=\, {W_j}^a H_a\bigr|_m + {W_j}^A \frac{\de}{\de f^A}\Bigr|_m \,.
\ee
The contraction above is
\be
\begin{split}
\tilde{\omega}_m(X,\, W_1,...,\,W_{k-1}) \,=\, &\tau^* \omega_m(X,\, W_1,...,\,W_{k-1}) \,+ \\
& X_{A_1\,...\,A_{k-1}} \left(\,\theta^{A_1} \wedge ... \wedge \theta^{A_{k-1}}\,\right)(W_1,...,\,W_{k-1}) \,+ \\
& X_{a_1\,A_2\,...\,A_{k-1}} \left(\, \dd x^{a_1} \wedge \theta^{A_2} \wedge ... \wedge \theta^{A_{k-1}} \,\right)(\,W_1,...,\,W_{k-1}) \,+\\
&\;\; \vdots \\
& X_{a_1\,...\,a_{k-2}\,A_{k-1}} \left(\, \dd x^{a_1} \wedge ... \dd x^{a_{k-2}} \wedge \theta^{A_{k-1}} \,\right)(W_1,...,\,W_{k-1}) \,.
\end{split}
\ee


\noindent First, consider all the $W_j$ to be vertical, that is, obeying ${W_j}^a \,=\, 0$.
The contraction above reads
\be
X_{A_1\,...\,A_{k-1}} \left(\,\theta^{A_1} \wedge ... \wedge \theta^{A_{k-1}}\,\right)(W_1,...,\,W_{k-1}) \,.
\ee
The condition that it has to vanish for all the $W_j$ gives
\be
X_{A_1\,...\,A_{k-1}} \,=\, 0 \,.
\ee
Let now $W_1$ be horizontal, i.e., $W_1 \,=\, {W_1}^a H_a$, and the remaining $W_j$ be vertical as in the previous case.
Here, the contraction reads
\be
X_{a_1\,A_2\,...\,A_{k-1}} \left(\, \dd x^{a_1} \wedge \theta^{A_2} \wedge ... \wedge \theta^{A_{k-1}} \,\right)(\,W_1,...,\,W_{k-1}) \,.
\ee
The fact that it has to be zero for all the $W_j$ gives 
\be
X_{a_1\,A_2\,...\,A_{k-1}} \,=\, 0 \,.
\ee
The same argument can be iterated by considering all the other possible choices for the $W_j$. 
One ends up with
\be
\begin{split}
X_{a_1\,a_2\,...\,A_{k-1}} \,&=\, 0 \,,\\
& \vdots \\
X_{a_1\,...\,a_{k-2}\,A_{k-1}} \,&=\, 0 \,,
\end{split}
\ee
proving that $X$ is tangent to $\mathcal{M}$ and, thus, that $\mathcal{M}$ is a $(k-1)$-coisotropic submanifold of a ${\Lambda^{k-1}}^\perp_R(\mathcal{M})$.
\end{proof}

\begin{remark}
Note that, unlike the classical case $k=2$ \cite{Gotay-CoisotropicEmbedding-1982}, the form $\tilde{\omega}$ remains non-degenerate throughout the entire ${\Lambda^{k-1}}^\perp_R(\mathcal{M})$. Consequently, there is no need to restrict to a tubular neighborhood of the zero-section of $\tau$ (as in \cite{Gotay-CoisotropicEmbedding-1982}) or to impose additional assumptions such as the flatness of $P$ (see \cite{Dubrovin-Giordano-Marmo-Simoni-PoissonPresymplectic-1993, Ciaglia-DiCosmo-Ibort-Marmo-Schiavone-Zampini-Peierls2-2022}) to ensure the non-degeneracy of $\tilde{\omega}$.
\end{remark}

\begin{remark}
Another key difference between the case $k>2$ analyzed here and the classical case $k=2$ concerns the uniqueness of the multisymplectic thickening.
Indeed, in \cite{Gotay-CoisotropicEmbedding-1982}, the symplectic thickening is not only guaranteed to exist but also proven to be unique (up to symplectomorphism).
In contrast, in the multisymplectic case, uniqueness generally cannot be established, as the following simple counterexample shows.

\noindent Consider the pre-$3$-plectic manifold $(\mathbb{R}^4,\, \omega)$ with coordinates 
\be
\left\{\, x^1,\, x^2,\, x^3,\, x^4 \,\right\} \,,
\ee
where
\be
\omega \,=\, \dd x^1 \wedge \dd x^2 \wedge \dd x^3 \,,
\ee
with kernel given by
\be
\mathrm{ker} \,\omega \,=\, \langle\, \left\{\, \frac{\de}{\de x^4} \,\right\} \,\rangle \,.
\ee
It can be embedded in both $(\mathbb{R}^5,\, \tilde{\omega})$ with coordinates
\be
\left\{\, x^1,\, x^2,\, x^3,\, x^4,\, x^5 \,\right\} \,,
\ee
where
\be
\tilde{\omega} \,=\, \dd x^1 \wedge \dd x^2 \wedge \dd x^3 + \dd x^1 \wedge \dd x^4 \wedge \dd x^5 \,,
\ee
and in $(\mathbb{R}^6,\, \tilde{\omega})$ with coordinates
\be
\left\{\, x^1,\, x^2,\, x^3,\, x^4,\, x^5,\, x^6 \,\right\} \,,
\ee
where
\be
\tilde{\omega} \,=\, \dd x^1 \wedge \dd x^2 \wedge \dd x^3 + \dd x^1 \wedge \dd x^4 \wedge \dd x^5 + \dd x^2 \wedge \dd x^4 \wedge \dd x^6 \,.
\ee 
They are both $3$-plectic manifolds, and one can easily check that the $(\mathbb{R}^4,\, \omega)$ is a $2$-coisotropic submanifold of both of them.

\noindent Therefore, one concludes that establishing an existence result—depending only on the (immaterial) choice of a complement, as in Gotay's construction—is the most canonical result one can achieve in the multisymplectic context, where uniqueness is not guaranteed.
\end{remark}

\section{Application: A Scalar Field Theory with Anisotropic Mass}
\label{Sec: Application}

As an application of \cref{Thm: coisotropic embedding}, we consider a scalar field theory on a 2-dimensional space-time in the multisymplectic Hamiltonian framework.
We refer to \cite{Ciaglia-DiCosmo-Ibort-Marmo-Schiavone-Zampini-Peierls1-2024} and references therein for the theory behind the multisymplectic approach to Hamiltonian field theories.

\noindent In the example we are considering, the space-time is the $2$-dimensional manifold $\mathbb{R}^2$ with coordinates $\left\{\,x,\,t\,\right\}$.
The multisymplectic phase space is the reduced dual of the first-order jet bundle of the trivial bundle $\pi \colon E \,=\, \mathbb{R} \times \mathbb{R}^2 \to \mathbb{R}^2$. 
It is a $5$-dimensional manifold $\mathcal{M}$ isomorphic to $\mathbb{R}^5$ with coordinates $\left\{\,x, t, u, \rho^x, \rho^t\right\}$, where $u$ represents the scalar field coordinate and $\rho^x$, $\rho^t$ are the so-called \textit{multimomenta} coordinates. 
The fields of the theory are the sections of the bundle $\mathcal{M} \to \mathbb{R}^2$, say
\be
\chi \;:\;\; (x,\,t) \mapsto (x,\,t,\, \phi(x,t),\, P^x(x,t),\, P^t(x,t)) \,.
\ee
Instead of the multisymplectic form
\be
\omega \,=\, \dd \rho^t \wedge \dd u \wedge \dd x + \dd\rho^x \wedge \dd u \wedge \dd t - \rho^t \dd \rho^t \wedge \dd x \wedge \dd t - \rho^x \dd \rho^x \wedge \dd x \wedge \dd t \,,
\ee
corresponding to the Hamiltonian 
\be
H \,=\, \frac{1}{2} \left(\, (\rho^x)^2 + (\rho^t)^2 \,\right) \,,
\ee
we will consider the pre-multisymplectic structure
\be
\omega = \dd \rho^x \wedge \dd u \wedge \dd t - \rho^x \dd \rho^x \wedge \dd x \wedge \dd t \,,
\ee
corresponding to a free scalar field in a singular limit where its mass becomes infinite in the time-like direction. 

\noindent The form $\omega$ is evidently closed, but it is degenerate.
Its kernel is easily seen to be
\be
K \,=\, \langle \left\{\, \frac{\de}{\de x} + \rho^x \frac{\de}{\de u},\, \frac{\de}{\de \rho^t} \,\right\} \rangle \,.
\ee

We now apply \cref{Thm: coisotropic embedding} for $k=3$. 
In this simple example, we will consider the connection $P$ to provide the following complement for $K$
\be
H \,=\, \langle \left\{\, \frac{\de}{\de t},\, \frac{\de}{\de u},\, \frac{\de}{\de \rho^x} \,\right\} \rangle \,.
\ee

\noindent The multisymplectic thickening is the manifold $\tilde{\mathcal{M}} = {\Lambda^2}^\perp_R(M)$, i.e., the bundle of transversal 2-forms w.r.t. $R \,=\, \mathbb{1} - P$.

\noindent For the connection chosen, the $1$-forms $\theta^A$ of \cref{Sec: Preliminaries} and \cref{Sec: Coisotropic embeddings of pre-multisymplectic manifolds} read $\{dx, d\rho^t\}$, namely
\be
P \,=\, \left(\, \frac{\de}{\de x} + \rho^x \frac{\de}{\de u} \,\right) \otimes \dd x + \frac{\de}{\de \rho^t} \otimes \dd \rho^t \,.
\ee
A transversal $2$-form is a $2$-form containing at least one $\theta^A$. 
This introduces $7$ auxiliary fields, which are the coordinates on the fibers of $\tilde{\mathcal{M}}$
\be
\left\{\, p_{xt},\, p_{xu},\, p_{x\rho^x},\, p_{x \rho^t},\, p_{\rho^t t},\, p_{\rho^t u},\, p_{\rho^t \rho^x} \,\right\} \,.
\ee
The manifold $\tilde{\mathcal{M}}$ thus reads $\mathbb{R}^{12}$ with coordinates 
\be
\left\{\,x,\,t, \,u,\, \rho^x,\, \rho^t,\, p_{xt},\, p_{xu},\, p_{x\rho^x},\, p_{x\rho^t},\, p_{\rho^t t},\, p_{\rho^t u},\, p_{\rho^t \rho^x} \,\right\} \,.
\ee

\noindent The new, non-degenerate $3$-form on $\tilde{\mathcal{M}}$ is $\tilde{\omega} = \tau^*\omega + d\Theta_0$, where $\Theta_0$ is the tautological 2-form:
\be
\begin{split}
\Theta_0 = &p_{xt} \dd x \wedge \dd t + p_{xu} \dd x \wedge \dd u + p_{x\rho^x} \dd x \wedge \dd \rho^x + p_{x\rho^t} \dd x \wedge \dd \rho^t + \\ 
+ & p_{\rho^t t} \dd \rho^t \wedge \dd t + p_{\rho^t u} \dd \rho^t \wedge \dd u + p_{\rho^t \rho^x} \dd \rho^t \wedge \dd \rho^x \,.
\end{split}
\ee

\subsection{Equations of motion}

Here, we will see the role of the regularization procedure in providing a fully determined system of equations of motion out of an implicit underdetermined one, showing that the regularization has the meaning of fixing the undetermined gauge in the equations of motion.

The equations of motion on $(\mathcal{M},\,\omega)$ are derived in the multisymplectic approach \cite{Ciaglia-DiCosmo-Ibort-Marmo-Schiavone-Zampini-Peierls1-2024} from the condition $\chi^\star(i_V \omega)=0$ for any vertical vector field on the fibration $\mathcal{M} \to \mathbb{R}^2$ 
\be
V\,=\,V_u \frac{\de}{\de u} + V_{\rho^x} \frac{\de}{\de \rho^x} + V_{\rho^t} \frac{\de}{\de \rho^t} \,.
\ee

\noindent A straightforward calculation shows that equations of motion read
\be
\begin{split}
\frac{\de P^x}{\de x} \,=\, 0 \,, \\
\frac{\de \phi}{\de x} \,=\, P^x \,.
\end{split}
\ee
As expected, since the system is degenerate, some degrees of freedom, the momentum $P^t$ in particular, are completely undetermined, namely, the system exhibits what is called a \textit{gauge symmetry}.
As a consequence, the field $\phi$ obeys
\be
\frac{\de^2 \phi}{\de x^2} \,=\, 0 \,,
\ee
but its evolution along the $t$-coordinate is completely free.

On the other hand, on the regularized manifold $(\tilde{\mathcal{M}},\, \tilde{\omega})$, the equations of motion provide a fully determined system where the gauge freedom is fixed.

\noindent In the regularized system, the additional degrees of freedom emerging will be considered as additional, non-physical fields playing a role that can be considered as the multisymplectic analogue of \textit{ghost fields} used in the symplectic framework.
Thus, on $(\tilde{\mathcal{M}},\, \tilde{\omega})$ equations of motion are derived from the condition $\tilde{\chi}^\star(i_{\tilde{V}} \tilde{\omega}) \,=\, 0$, for $\tilde{\chi}$ being a section of the bundle $\tilde{\mathcal{M}} \to \mathbb{R}^2$, say
\be
\tilde{\chi} \;:\;\; (x,\,t) \mapsto (x,\,t,\, \phi(x,t),\, P^x(x,t),\,P^t(x,t), \Pi_{xt}(x,t),\, \Pi_{xu}(x,t),...,\,\Pi_{\rho^t\rho^x}(x,t)) \,,
\ee
and for any vertical vector field $\tilde{V}$.

\begin{remark}
We are well aware that a rigorous application of this procedure raises a fundamental structural question: does the multisymplectic thickening $\tilde{\mathcal{M}}$ of a multisymplectic phase space $\mathcal{M}$ (which is, by construction, the reduced dual of a jet bundle) inherit a similar structure? In the corresponding Lagrangian formalism, this is equivalent to asking under which conditions the thickening of a jet bundle is itself a jet bundle.

\noindent A thorough geometric analysis of this problem is required to ensure the global consistency of the regularized field theory. However, such an investigation lies beyond the scope of this paper, whose main goal is to establish the coisotropic embedding theorem itself. This structural problem will be addressed in a future work. For the purely illustrative purpose of this example, we proceed by assuming that a consistent description of the regularized fields as sections $\tilde{\chi}$ of a suitable bundle $\tilde{\mathcal{M}} \to \mathbb{R}^2$ exists.
\end{remark}

\noindent A long but straightforward computation shows that in this case, the equations of motion read
\be
\begin{split}
\frac{\de \phi}{\de x} \,&=\, P^x \,,\\
\frac{\de \phi}{\de t} \,&=\, 0 \,,\\
\frac{\de P^x}{\de x} \,&=\, 0 \,,\\
\frac{\de P^x}{\de t} \,&=\, 0 \,,\\
\frac{\de P^t}{\de x} \,&=\, 0 \,,\\
\frac{\de P^t}{\de t} \,&=\, 0 \,,\\
\end{split}
\ee
that amount to
\be
\begin{split}
\frac{\de^2 \phi}{\de x^2} \,=\, 0 \,, \\
\frac{\de \phi}{\de t} \,=\, 0 \,,
\end{split}
\ee
where now, the dynamics of $\phi$ along the direction $t$ has been fixed to be constant.

\noindent Additionally, from $\tilde{\chi}^\star(i_{\tilde{V}} \tilde{\omega}) \,=\, 0$ one gets a set of "dynamical" equations for the additional, un-physical fields $\Pi_{xt},...,\,\Pi_{\rho^t \rho^x}$, that we avoid to write explicitely since they will be neglected.

This example illustrates how the coisotropic embedding procedure takes a pre-multisymplectic system with a non-trivial kernel, and, thus, with a gauge symmetry, and extends it to a larger, non-degenerate multisymplectic system where all the gauge freedoms are fixed.

\section*{Conclusions and Future Plans}

In this paper, we proved a coisotropic embedding theorem for pre-multisymplectic manifolds. 
We constructed a suitable multisymplectic thickening of any given pre-$k$-plectic manifold in which it is embedded as a $(k-1)$-coisotropic submanifold. 
It is the bundle of transversal $(k-1)$-forms on the pre-multisymplectic manifold associated with some connection.
This embedding may be employed for a finite-dimensional approach to the regularization techniques listed in the introduction, which traditionally require handling infinite-dimensional spaces.

\noindent Building on these results, several directions for future research emerge, as we are going to explain.

\paragraph{Applications to Field Theories}
We will apply the coisotropic embedding theorem proved to Field Theories, particularly in constructing a Poisson bracket on the space of solutions, as we already did in \cite{Ciaglia-DiCosmo-Ibort-Marmo-Schiavone-Zampini-Peierls1-2024, Ciaglia-DiCosmo-Ibort-Marmo-Schiavone-Zampini-Peierls2-2022, Ciaglia-DiCosmo-Ibort-Marmo-Schiavone-Zampini-Peierls3-2022}. 
We will study the relation between the two approaches and the potential advantages of the one based on the multisymplectic coisotropic embedding theorem proved in this manuscript.

\paragraph{Application to the Integrability of Implicit Partial Differential Equations}
Implicit partial differential equations can often be formulated using pre-multisymplectic structures. 
The regularization method introduced in this paper allows for the definition of a multisymplectic structure on an extended manifold, enabling (under some conditions that must be studied\footnote{A first, obvious, one is: if the pre-multisymplectic manifold one starts with is a first-order jet bundle (or its reduced dual), under which conditions the multisymplectic thickening is a first-order jet bundle as well?}) the formulation of an explicit PDE. 
Future work will investigate:
\begin{itemize}
    \item The correspondence between solutions of the original implicit PDE and the explicit PDE on the multisymplectic thickening.
    \item The conditions under which the integrability of the explicit PDE ensures the integrability of the implicit one.
    \item The development of tools to apply this framework to specific classes of implicit PDEs, thereby improving the understanding of their geometric and analytic properties.
\end{itemize}

\paragraph{Extending Local Results to Global Contexts}
Several results proven by the author, including those in \cite{Ciaglia-DiCosmo-Ibort-Marmo-Schiavone-Zampini-Symmetry-2022, Schiavone-InverseProblemElectrodynamics-2024, Schiavone-InverseProblemImplicit-2024}, rely on a local formulation of Field Theories as pre-symplectic Hamiltonian systems. 
Using the coisotropic embedding theorem, we aim to extend these results to global contexts by avoiding passing through the above-mentioned pre-symplectic Hamiltonian system, which is only locally defined close to some Cauchy hypersurface. 

\section*{Acknowledgements}

I am indebted to prof. A. Ibort for making me aware of the existence of the coisotropic embedding theorem and its potential applications as well as for the huge amount of inspiring discussions of the last years.
I am also indebted to prof. G. Marmo for encouraging me to think of alternative ways to depict the geometry behind Gotay's coisotropic embedding theorem.
I acknowledge financial support from Next Generation EU through the project 2022XZSAFN – PRIN2022 CUP: E53D23005970006.
I am a member of the GNSGA (Indam). 

\section*{Conflict of interests/Competing interests}

The author has no conflict of interests or competing interests to disclose.

\bibliographystyle{alpha}
\bibliography{Biblio}

\end{document}